\documentclass[a4paper,UKenglish,cleveref]{lipics-v2021}

\pdfoutput=1

\hypersetup{
    colorlinks=true,
    linkcolor=black,
    citecolor=black,
    filecolor=black,
    urlcolor=black,
}
\title{Predicative Aspects of Order Theory in Univalent~Foundations}

\titlerunning{Predicative Aspects of Order Theory in UF}

\author{Tom de Jong}{University of Birmingham, United Kingdom \and \url{https://www.cs.bham.ac.uk/~txd880}}
{t.dejong@pgr.bham.ac.uk}{https://orcid.org/0000-0003-1585-3172}{}

\author{Mart\'in H\"otzel Escard\'o}{University of Birmingham, United Kingdom \and \url{https://www.cs.bham.ac.uk/~mhe}}
{m.escardo@cs.bham.ac.uk}{https://orcid.org/0000-0002-4091-6334}{}

\authorrunning{T. de Jong and M.\,H. Escard\'o}

\Copyright{Tom de Jong and Mart\'in H. Escard\'o}



\category{} %optional, e.g. invited paper

\nolinenumbers %uncomment to disable line numbering

\hideLIPIcs  %uncomment to remove references to LIPIcs series (logo, DOI, ...), e.g. when preparing a pre-final version to be uploaded to arXiv or another public repository

\EventEditors{Naoki Kobayashi}
\EventNoEds{1}
\EventLongTitle{6th International Conference on Formal Structures for Computation and Deduction (FSCD 2021)}
\EventShortTitle{FSCD 2021}
\EventAcronym{FSCD}
\EventYear{2021}
\EventDate{July 17--24, 2021}
\EventLocation{Buenos Aires, Argentina (Virtual Conference)}
\EventLogo{}
\SeriesVolume{195}
\ArticleNo{8}
\usepackage{mymacros}

\usepackage{tikz-cd}

\begin{document}

\maketitle

\begin{abstract}
  We investigate predicative aspects of order theory in constructive univalent
  foundations. By~predicative and constructive, we respectively mean that we do
  not assume Voevodsky's propositional resizing axioms or excluded middle.
  Our work complements existing work on predicative mathematics by
  exploring what \emph{cannot} be done predicatively in univalent foundations.
  Our first main result is that nontrivial (directed or bounded) complete posets
  are necessarily large. That is, if such a nontrivial poset is small, then weak
  propositional resizing holds. It is possible to derive full propositional
  resizing if we strengthen nontriviality to positivity. The~distinction between
  nontriviality and positivity is analogous to the distinction between
  nonemptiness and inhabitedness.
  We prove our results for a general class of posets, which includes directed
  complete posets, bounded complete posets and sup-lattices, using a technical
  notion of a \deltacomplete{\V} poset.
  We also show that nontrivial locally small \deltacomplete{\V} posets
  necessarily lack decidable equality. Specifically, we derive weak excluded
  middle from assuming a nontrivial locally small \deltacomplete{\V} poset
  with decidable equality. Moreover, if we assume positivity instead of
  nontriviality, then we can derive full excluded middle.
  Secondly, we show that each of Zorn's lemma, Tarski's greatest fixed point
  theorem and Pataraia's lemma implies propositional resizing. Hence, these
  principles are inherently impredicative and a predicative development of order
  theory must therefore do without them.
  Finally, we clarify, in our predicative setting, the relation between the
  traditional definition of sup-lattice that requires suprema for all subsets and
  our definition that asks for suprema of all small families.
\end{abstract}

\section{Introduction}
We investigate predicative aspects of order theory in constructive univalent
foundations. By~predicative and constructive, we respectively mean that we do
not assume Voevodsky's propositional resizing
axioms~\cite{Voevodsky2011,Voevodsky2015} or excluded middle.
Our work is situated in our larger programme of developing domain theory
constructively and predicatively in univalent foundations. In~previous
work~\cite{deJongEscardo2021}, we showed how to give a constructive and
predicative account of many familiar constructions and notions in domain theory,
such as Scott's \(D_\infty\) model of untyped \(\lambda\)-calculus and the
theory of continuous dcpos. The present work complements this and other existing
work on predicative mathematics
(e.g.~\cite{AczelRathjen2010,Sambin1987,CoquandEtAl2003}) by exploring what
\emph{cannot} be done predicatively, as
in~\cite{Curi2010a,Curi2010b,Curi2015,Curi2018,CuriRathjen2012}. We do so by
showing that certain statements crucially rely on resizing axioms in the sense
that they are equivalent to them. Such arguments are important in
constructive mathematics.  For example, the constructive failure of
trichotomy on the real numbers is shown~\cite{BridgesRichman1987} by reducing it
to a nonconstructive instance of excluded middle.

Our first main result is that nontrivial (directed or bounded) complete posets
are necessarily large. In~\cite{deJongEscardo2021} we observed that all our
examples of directed complete posets have large carriers. We show here that this
is no coincidence, but rather a necessity, in the sense that if such a
nontrivial poset is small, then weak propositional resizing holds. It is
possible to derive full propositional resizing if we strengthen nontriviality to
positivity in the sense of~\cite{Johnstone1984}. The distinction between
nontriviality and positivity is analogous to the distinction between
nonemptiness and inhabitedness.
We prove our results for a general class of posets, which includes directed
complete posets, bounded complete posets and sup-lattices, using a technical
notion of a \deltacomplete{\V} poset.
We also show that nontrivial locally small \deltacomplete{\V} posets necessarily
lack decidable equality. Specifically, we can derive weak excluded middle from
assuming the existence of a nontrivial locally small \deltacomplete{\V} poset
with decidable equality. Moreover, if we assume positivity instead of
nontriviality, then we can derive full excluded middle.

Secondly, we prove that each of Zorn's lemma, Tarski's greatest fixed point theorem and
Pataraia's lemma implies propositional resizing. Hence, these principles are
inherently impredicative and a predicative development of order theory in
univalent foundations must thus forgo them.

Finally, we clarify, in our predicative setting, the relation between the
traditional definition of sup-lattice that requires suprema for all subsets and
our definition that asks for suprema of all small families. This is important in
practice in order to obtain workable definitions of dcpo, sup-lattice, etc.\
in the context of predicative univalent mathematics.

Our foundational setup is the same as in~\cite{deJongEscardo2021}, meaning that
our work takes places in intensional Martin-L\"of Type Theory and adopts the
univalent point of view~\cite{HoTTBook}. This~means that we work with the
stratification of types as singletons, propositions (or subsingletons or truth
values), sets, {1-groupoids}, etc., and that we work with univalence. At present,
higher inductive types other than propositional truncation are not needed. Often
the only consequences of univalence needed here are functional and propositional
extensionality.  An exception is \cref{sec:size-and-univalence}. Full details of
our univalent type theory are given at the start
of~\cref{sec:foundations-and-size-matters}.

\paragraph*{Related work}

Curi investigated the limits of predicative mathematics
in CZF~\cite{AczelRathjen2010} in a series of
papers~\cite{Curi2010a,Curi2010b,Curi2015,Curi2018,CuriRathjen2012}.
In particular, Curi shows (see~\cite[Theorem~4.4 and
  Corollary~4.11]{Curi2010a}, \cite[Lemma~1.1]{Curi2010b} and
  \cite[Theorem~2.5]{Curi2015}) that CZF cannot prove that various nontrivial
posets, including sup-lattices, dcpos and frames, are small. This result is
obtained by exploiting that CZF is consistent with the anti-classical
generalized uniformity principle
GUP~\cite[Theorem~4.3.5]{vandenBerg2006}.
Our related \cref{nontrivial-impredicativity,positive-impredicativity} is of a
different nature in two ways.
Firstly, our theorem is in the spirit of reverse constructive
mathematics~\cite{Ishihara2006}: Instead of showing that GUP implies that there
are no non-trivial small dcpos, we show that the existence of a non-trivial
small dcpo is \emph{equivalent} to weak propositional resizing, and that the
existence of a positive small dcpo is \emph{equivalent} to full propositional
resizing. Thus, if we wish to work with small dcpos, we are forced to assume
resizing axioms.
Secondly, we work in univalent foundations rather
than CZF.  This may seem a superficial difference, but a number of
arguments in Curi's papers~\cite{Curi2015,Curi2018} crucially rely on
set-theoretical notions and principles such as transitive set,
set-induction, weak regular extension axiom wREA, which cannot even be
formulated in the underlying type theory of univalent
foundations.
Moreover, although Curi claims that the arguments of~\cite{Curi2010a,Curi2010b}
can be adapted to some version of Martin-L\"of Type Theory, it is presently
not known whether there is any model of univalent foundations which validates
GUP.

\paragraph*{Organization}
\emph{\cref{sec:foundations-and-size-matters}}: Foundations and size matters,
including impredicativity, relation to excluded middle, univalence and closure
under embedded retracts.
\emph{\cref{sec:large-posets}}: Nontrivial and positive \deltacomplete{\V}
posets and reductions to impredicativity and excluded middle.
\emph{\cref{sec:maximal-and-fixed-points}}: Predicative invalidity of Zorn's
lemma, Tarski's fixed point theorem and Pataraia's lemma.
\emph{\cref{sec:families-and-subsets}}: Comparison of completeness w.r.t.\
families and w.r.t.\ subsets.
\emph{\cref{sec:conclusion}}: Conclusion and future work.

\section{Foundations and Size Matters}\label{sec:foundations-and-size-matters}
We work with a subset of the type theory described in~\cite{HoTTBook} and we
mostly adopt the terminological and notational conventions
of~\cite{HoTTBook}. We include \(+\)~(binary sum), \(\Pi\)~(dependent
products), \(\Sigma\)~(dependent sum), \(\Id\) (identity type), and inductive
types, including~\(\Zero\)~(empty type), \(\One\)~(type with exactly one element
\(\star : \One\)), \(\Nat\)~(natural numbers).
We assume a universe \(\U_0\) and two operations: for every universe \(\U\) a
successor universe \(\U^+\) with \(\U : \U^+\), and for every two universes
\(\U\) and \(\V\) another universe \(\U \sqcup \V\) such that for any universe
\(\U\), we have \(\U_0 \sqcup \U \equiv \U\) and \(\U \sqcup \U^+ \equiv
\U^+\). Moreover, \((-)\sqcup(-)\) is idempotent, commutative, associative, and
\((-)^+\) distributes over \((-)\sqcup(-)\). We write
\(\U_1 \colonequiv \U_0^+\), \(\U_2 \colonequiv \U_1^+, \dots\) and so on.  If
\(X : \U\) and \(Y : \V\), then \({X + Y} : \U \sqcup \V\) and if \(X : \U\) and
\(Y : X \to \V\), then the types \(\Sigma_{x : X} Y(x)\) and
\(\Pi_{x : X} Y(x)\) live in the universe \(\U \sqcup \V\); finally,
if~\(X : \U\) and \(x,y : X\), then \(\Id_{X}(x,y) : \U\). The type of natural
numbers \(\Nat\) is assumed to be in \(\U_0\) and we postulate that we have
copies \(\Zero_{\U}\) and \(\One_{\U}\) in every universe \(\U\). We assume
function extensionality and propositional extensionality tacitly, and univalence
explicitly when needed. Finally, we use a single higher inductive type: the
propositional truncation of a type \(X\) is denoted by \(\squash*{X}\) and we
write \(\exists_{x : X}Y(x)\) for \(\squash*{\sum_{x : X}Y(x)}\).

\subsection{The Notion of Size}
We introduce the fundamental notion of a type having a certain size and specify
the impredicativity axioms under consideration~(\cref{sec:impred-and-em}). We
also note the relation to excluded middle~(\cref{sec:impred-and-em}) and
univalence~(\cref{sec:size-and-univalence}). Finally
in~\cref{sec:size-and-retracts} we review embeddings and sections and establish
our main technical result on size, namely that having a certain size is closed
under retracts whose sections are embeddings.

\begin{definition}[Size,
   \href{https://www.cs.bham.ac.uk/~mhe/agda-new/UF-Size.html\#_has-size_}
         {\texttt{UF-Slice.html}}
    in \cite{TypeTopology}]
    A type \(X\) in a universe \(\U\) is said to \emph{have~size}~\(\V\) if it
    is equivalent to a type in the universe \(\V\). That is,
    \({X \hassize \V} \colonequiv \sum_{Y : \V} \pa*{Y \simeq X}\).
\end{definition}

\subsection{Impredicativity and Excluded Middle}
\label{sec:impred-and-em}
We consider various impredicativity axioms and their relation to (weak) excluded
middle. The definitions and propositions below may be found in
\cite[Section~3.36]{Escardo2020}, so proofs are omitted here.

\begin{definition}[Impredicativity axioms]
  \hfill
  \begin{enumerate}[(i)]
  \item By \emph{Propositional-\(\text{Resizing}_{\U,\V}\)} we mean the
    assertion that every proposition \(P\) in a universe \(\U\) has size \(\V\).
  \item The type of all propositions in a universe \(\U\) is denoted by
    \(\Omega_{\U}\). Observe that \(\Omega_{\U} : \U^+\).  We write
    \emph{\(\Omegaresizing{\U}{\V}\)} for the assertion that the type
    \(\Omega_{\U}\) has size \(\V\).
  \item The type of all \(\lnot\lnot\)-stable propositions in a universe \(\U\)
    is denoted by \(\Omeganotnot{\U}\), where a proposition \(P\) is
    \emph{\(\lnot\lnot\)-stable} if \(\lnot\lnot P\) implies \(P\).
    By \emph{\(\Omeganotnotresizing{\U}{\V}\)} we mean the assertion that the
    type \(\Omeganotnot{\U}\) has size \(\V\).
  \item For the particular case of a single universe, we write
    \(\Omegaresizingalt{\U}\) and \(\Omeganotnotresizingalt{\U}\) for the
    respective assertions that \(\Omega_{\U}\) has~size~\(\U\) and
    \(\Omeganotnot{\U}\) has~size~\(\U\).
  \end{enumerate}
\end{definition}

\begin{proposition}
  \hfill
  \begin{enumerate}[(i)]
  \item The principle \(\Omegaresizing{\U}{\V}\) implies
    \(\Propresizing{\U}{\V}\) for every two universes \(\U\) and \(\V\).
  \item The conjunction of \(\Propresizing{\U}{\V}\) and \(\Propresizing{\V}{\U}\)
  implies \(\Omegaresizing{\U}{\V^+}\) for every two universes \(\U\) and \(\V\).
  \end{enumerate}
\end{proposition}
It is possible to define a weaker variation of propositional resizing for
\(\lnot\lnot\)-stable propositions only (and derive similar connections), but we
don't have any use for it in this paper.

\begin{definition}[(Weak) excluded middle]
  \hfill
  \begin{enumerate}[(i)]
  \item \emph{Excluded middle} in a universe \(\U\) asserts that for every
    proposition \(P\) in \(\U\) either \(P\)~or~\(\lnot P\) holds.
  \item \emph{Weak excluded middle} in a universe \(\U\) asserts that for every
    proposition \(P\) in \(\U\) either \(\lnot P\) or \(\lnot\lnot P\) holds.
  \end{enumerate}
\end{definition}
We note that weak excluded middle says precisely that \(\lnot\lnot\)-stable
propositions are decidable and is equivalent to de~Morgan's Law.

\begin{proposition}
  Excluded middle implies impredicativity. Specifically,
  \begin{enumerate}[(i)]
  \item Excluded middle in \(\U\) implies \(\Omegaresizing{\U}{\U_0}\).
  \item Weak excluded middle in \(\U\) implies
    \(\Omeganotnotresizing{\U}{\U_0}\).
  \end{enumerate}
\end{proposition}

\subsection{Size and Univalence}
\label{sec:size-and-univalence}
Assuming univalence we can prove that \(\Propresizing{\U}{\V}\) and
\(\Omegaresizing{\U}{\V}\) are subsingletons. More generally, univalence allows
us to prove that the statement that \(X\) has size \(\V\) is a proposition,
which is needed in \cref{sec:unspecified}.

\begin{proposition}[cf.\ \texttt{has-size-is-subsingleton} in \cite{Escardo2020}]
  \label{has-size-is-prop}
  If \(\V\) and \(\U \sqcup \V\) are univalent universes, then \(X \hassize \V\)
  is a proposition for every \(X : \U\).
\end{proposition}
The converse also holds in the following form.
\begin{proposition}\label{has-size-univalence}
  The type \(X \hassize \U\) is a proposition for
  every \(X : \U\) if and only if \(\U\) is a univalent universe.
\end{proposition}
\begin{proof}
  Note that \({X \hassize \U}\) is \({\sum_{Y : \U}Y\simeq X}\), so this can be
  found in \cite[Section~3.14]{Escardo2020}. \qedhere
\end{proof}

\subsection{Size and Retracts}
\label{sec:size-and-retracts}
We show our main technical result on size here, namely that having a size is
closed under retracts whose sections are embeddings.

\begin{definition}[Sections, retractions and embeddings]
  \hfill
  \begin{enumerate}[(i)]
  \item A \emph{section} is a map \(s : X \to Y\) together with a left inverse
    \(r : Y \to X\), i.e.\ the maps satisfy \(r \circ s \sim \id\).  We call
    \(r\) the \emph{retraction} and say that \(X\) is a \emph{retract} of \(Y\).
  \item A function \(f : X \to Y\) is an embedding if the map
    \(\ap_f : \pa*{x = y} \to \pa*{f(x) = f(y)}\) is an equivalence for every
    \(x,y : X\). (See~\cite[Definition~4.6.1(ii)]{HoTTBook}.)
  \item A \emph{section-embedding} is a section \(s : {X \to Y}\) that moreover
    is an embedding. We also say that \(X\) is an \emph{embedded retract} of
    \(Y\).
  \end{enumerate}
\end{definition}
We recall the following facts about embeddings and sections.
\begin{lemma}\label{embeddings-sections-lemmas}
  \hfill
  \begin{enumerate}[(i)]
  \item A function \(f : X \to Y\) is an \emph{embedding} if and only if all its
    fibres are subsingletons, i.e.\
    \(\prod_{y :
      Y}\issubsingleton\pa*{\fib_f(y)}\).
    (See~\cite[Proof~of~Theorem~4.6.3]{HoTTBook}.)
  \item If every section is an embedding, then every type is a set.
    (See~\cite[Remark~3.11(2)]{Shulman2016}.)
  \item Sections to sets are embeddings.
    (See~\cite[\texttt{lc-maps-into-sets-are-embeddings}]{Escardo2020}.)
  \end{enumerate}
\end{lemma}

In phrasing our results it is helpful to extend the notion of size from
types to functions.

\begin{definition}[Size (for functions),
   \href{https://www.cs.bham.ac.uk/~mhe/agda-new/UF-Size.html\#_Has-size_}
         {\texttt{UF-Slice.html}}
         in \cite{TypeTopology}]
         A function \(f : X \to Y\) is said to \emph{have size} \(\V\) if every
         fibre has size \(\V\).
\end{definition}

\begin{lemma}[cf.\
  \href{https://www.cs.bham.ac.uk/~mhe/agda-new/UF-Size.html\#_Has-size_}
  {\texttt{UF-Slice.html}} in \cite{TypeTopology}]
  \label{function-has-size-lemmas}\hfill
  \begin{enumerate}[(i)]
  \item A type \(X\) has size \(\V\) if and only if the unique map
    \(X \to \One_{\U_0}\) has size \(\V\).
  \item If \(f : X \to Y\) has size \(\V\) and \(Y\) has size \(\V\), then so
    does \(X\).
  \item If \(s : X \to Y\) is a section-embedding and \(Y\) has size \(\V\),
    then \(s\) has size \(\V\) too, regardless of the size of \(X\).
  \end{enumerate}
\end{lemma}
\begin{proof}
  The first two claims follow from the fact that for any map \(f : X \to Y\) we
  have an equivalence \(X \simeq \sum_{y : Y}\fib_f(y)\)
  (see~\cite[Lemma~4.8.2]{HoTTBook}). For the third claim, suppose that
  \(s : X \to Y\) an embedding with retraction \(r : Y \to X\). By the second
  part of the proof of Theorem~3.10 in~\cite{Shulman2016}, we have
  \(\fib_s(y) \simeq \squash*{s(r(y))=y}\), from which the claim follows.
\end{proof}

\begin{lemma}\label{size-retract}\hfill
  \begin{enumerate}[(i)]
  \item If \(X\) is an embedded retract of \(Y\) and \(Y\) has size \(\V\), then
    so does \(X\).
  \item If \(X\) is a retract of a set \(Y\) and \(Y\) has size \(\V\), then so
    does \(X\).
  \end{enumerate}
\end{lemma}
\begin{proof}
  The first statement follows from (ii) and (iii)
  of~\cref{function-has-size-lemmas}. The second follows from the first
  and item~(iii) of~\cref{embeddings-sections-lemmas}.
\end{proof}

\section{Large Posets Without Decidable Equality}
\label{sec:large-posets}
We show that constructively and predicatively many structures from order theory
(directed complete posets, bounded complete posets, sup-lattices) are
necessarily large and necessarily lack decidable equality. We capture these
structures by a technical notion of a \deltacomplete{\V} poset
in~\cref{sec:delta-complete-posets}. In~\cref{sec:nontrivial-and-positive} we
define when such structures are nontrivial and introduce the constructively
stronger notion of positivity. \cref{sec:retract-lemmas} and
\cref{sec:reductions} contain the two fundamental technical lemmas and the main
theorems, respectively. Finally, \cref{sec:unspecified} considers alternative
formulations of being nontrivial and positive that ensure that these notions are
properties, as opposed to data and shows how the main theorems remain valid,
assuming univalence.

\subsection{\texorpdfstring{\(\delta_{\V}\)}{delta\_V}-complete Posets}
\label{sec:delta-complete-posets}
We start by introducing a class of weakly complete posets that we call
\deltacomplete{\V} posets. The notion of a \deltacomplete{\V} poset is a
technical and auxiliary notion sufficient to make our main theorems go
through. The important point is that many familiar structures (dcpos, bounded
complete posets, sup-lattices) are \deltacomplete{\V} posets
(see~\cref{examples-of-delta-complete-posets}).

\begin{definition}[\deltacomplete{\V} poset,
  \(\delta_{x,y,P}\),
  \(\bigvee \delta_{x,y,P}\)]
  A \emph{poset} is a type \(X\) with a subsingleton-valued binary relation
  \({\below}\) on \(X\) that is reflexive, transitive and antisymmetric.  It is
  not necessary to require \(X\) to be a set, as this follows from the other
  requirements.
  A poset \((X,{\below})\) is \emph{\(\delta_\V\)-complete} for a
  universe~\(\V\) if for every pair of elements \(x,y : X\) with \(x \below y\)
  and every subsingleton \(P\) in \(\V\), the family
  \begin{align*}
    \delta_{x,y,P} : 1 + P &\to X \\
    \inl(\star) &\mapsto x; \\
    \inr(p) &\mapsto y;
  \end{align*}
  has a supremum \(\bigvee \delta_{x,y,P}\) in \(X\).
\end{definition}
% N.B. In the above definition, the carrier of the poset and the values of the
% partial order can live in arbitrary universes which may or may not be different
% from \(\V\).
%
\begin{remark}[Every poset is \deltacomplete{\V}, classically]
  \label{classically-every-poset-is-delta-complete}
  Consider a poset \((X,\below)\) and a pair of elements \(x \below y\). If
  \(P : \V\) is a decidable proposition, then we can define the supremum of
  \(\delta_{x,y,P}\) by case analysis on whether \(P\) holds or not. For if it
  holds, then the supremum is \(y\), and if it does not, then the supremum is
  \(x\). Hence, if excluded middle holds in \(\V\), then the family
  \(\delta_{x,y,P}\) has a supremum for every \(P : \V\). Thus, if excluded
  middle holds in \(\V\), then every poset (in any universe) is
  \deltacomplete{\V}.
\end{remark}
The above remark naturally leads us to ask whether the converse also holds,
i.e.\ if every poset is \deltacomplete{\V}, does excluded middle in \(\V\)
hold?  As far as we know, we can only get weak excluded middle in \(\V\), as we
will later see in~\cref{Two-is-not-delta-complete}.
This proposition also shows that in the absence of excluded middle, the notion
of \(\delta_{\V}\)-completeness isn't trivial. For now, we focus on the fact
that, also constructively and predicatively, there are many examples of
\deltacomplete{\V} posets.

\begin{examples}\hfill
  \label{examples-of-delta-complete-posets}
  \begin{enumerate}[(i)]
  \item Every \(\V\)-sup-lattices is \deltacomplete{\V}. That is, if a poset \(X\)
    has suprema for all families \(I \to X\) with \(I\) in the universe
    \(\V\), then \(X\) is \deltacomplete{\V}.
  \item The \(\V\)-sup-lattice \(\Omega_\V\) is \deltacomplete{\V}.  The type
    \(\Omega_{\V}\) of propositions in \(\V\) is a \(\V\)-sup-lattice with the
    order given by implication and suprema by existential quantification. Hence,
    \(\Omega_{\V}\) is \deltacomplete{\V}. Specifically, given propositions
    \(Q\), \(R\) and \(P\), the supremum of \(\delta_{Q,R,P}\) is given by
    \(Q \vee \pa*{R \times P}\).
  \item The \(\V\)-powerset \(\powerset_{\V}(X) \colonequiv X \to \Omega_{\V}\)
    of a type \(X\) is \deltacomplete{\V}. Note that \(\powerset_{\V}(X)\) is
    another example of a \(\V\)-sup-lattice (ordered by subset inclusion and
    with suprema given by unions) and hence \deltacomplete{\V}.
  \item Every \(\V\)-bounded complete posets is \deltacomplete{\V}. That is, if
    \((X,\below)\) is a poset with suprema for all bounded families \(I \to X\)
    with \(I\) in the universe \(\V\), then \((X,\below)\) is
    \deltacomplete{\V}.
    A family \(\alpha : I \to X\) is bounded if there exists some \(x : X\) with
    \(\alpha(i) \below x\) for every \(i : I\). For example, the family
    \(\delta_{x,y,P}\) is bounded by \(y\).
  \item Every \(\V\)-directed complete poset (dcpo) is \deltacomplete{\V}, since
    the family \(\delta_{x,y,P}\) is directed.  We note
    that~\cite{deJongEscardo2021} provides a host of examples of \(\V\)-dcpos.
  \end{enumerate}
\end{examples}

\subsection{Nontrivial and Positive Posets}
\label{sec:nontrivial-and-positive}
In \cref{classically-every-poset-is-delta-complete} we saw that if we can decide
a proposition \(P\), then we can define \(\bigvee \delta_{x,y,P}\) by case
analysis. What about the converse? That is, if \(\delta_{x,y,P}\) has a supremum
and we know that it equals \(x\) or \(y\), can we then decide \(P\)?  Of course,
if \(x = y\), then \(\bigvee \delta_{x,y,P} = x = y\), so we don't learn
anything about \(P\). But what if add the assumption that \(x \neq y\)? It turns
out that constructively we can only expect to derive decidability of \(\lnot P\)
in that case. This is due to the fact that \(x \neq y\) is a negated
proposition, which is rather weak constructively, leading us to later define
(see~\cref{def:strictly-below}) a constructively stronger notion for elements of
\deltacomplete{\V} posets.

\begin{definition}[Nontrivial]
  A poset \((X,\below)\) is \emph{nontrivial} if we have designated \(x,y : X\)
  with \(x \below y\) and \(x \neq y\).
\end{definition}

\begin{lemma}
  \label{delta-sup-weak-em}
  Let \((X,{\below},x,y)\) be a nontrivial poset. We have the
  following implications for every proposition \(P : \V\):
  \begin{enumerate}[(i)]
  \item\label{delta-sup-weak-em-1} if the supremum of \(\delta_{x,y,P}\) exists and
    \(x = \bigvee \delta_{x,y,P}\), then \(\lnot P\) is the case.
  \item\label{delta-sup-weak-em-2} if the supremum of \(\delta_{x,y,P}\) exists and
    \(y = \bigvee \delta_{x,y,P}\), then \(\lnot\lnot P\) is the case.
  \end{enumerate}
\end{lemma}
\begin{proof}
  Let \(P : \V\) be an arbitrary proposition. For (i), suppose that
  \(x = \bigvee \delta_{x,y,P}\) and assume for a contradiction that we have
  \(p : P\). Then
  \( y \equiv \delta_{x,y,P}(\inr(p)) \below \bigvee \delta_{x,y,P} = x, \)
  which is impossible by antisymmetry and our assumptions that \(x \below y\)
  and \(x \neq y\).
  For (ii), suppose that \(y = \bigvee \delta_{x,y,P}\) and assume for a
  contradiction that \(\lnot P\) holds. Then \(x = \bigvee \delta_{x,y,P} = y\),
  contradicting our assumption that \(x \neq y\). \qedhere
\end{proof}

\begin{proposition}[cf.\ Section 4 of~\cite{deJongEscardo2021}]
  \label{Two-is-not-delta-complete}
  Let \(\Two\) be the poset with exactly two elements \(0 \below 1\).
  If~\(\Two\)~is \deltacomplete{\V}, then weak excluded middle in \(\V\) holds.
\end{proposition}
\begin{proof}
  Suppose that \(\Two\) were \deltacomplete{\V} and let \(P : \V\) be an
  arbitrary subsingleton. We must show that \(\lnot P\) is decidable. Since
  \(\Two\) has exactly two elements, the supremum \(\bigvee \delta_{0,1,P}\)
  must be \(0\) or \(1\). But then we apply \cref{delta-sup-weak-em} to get
  decidability of \(\lnot P\).
\end{proof}
That the conclusion of the implication in
\cref{delta-sup-weak-em}\eqref{delta-sup-weak-em-2} cannot be strengthened to
say that \(P\) is the case is shown by the following observation.
%
% That we can't improve \cref{delta-sup-weak-em} to decidability of \(P\) is shown
% by the following observation.
\begin{proposition}
  \label{delta-sup-em}
  Recall~\cref{examples-of-delta-complete-posets}, which show that
  \(\Omega_{\V}\) is \deltacomplete{\V}. If for every two propositions \(Q\) and
  \(R\) with \(Q \below R\) and \(Q \neq R\) we have that the equality
  \(R = \bigvee \delta_{Q,R,P}\) in \(\Omega_{\V}\) implies \(P\) for every
  proposition \(P : \V\), then excluded middle in \(\V\) follows.
\end{proposition}
\begin{proof}
  Assume the hypothesis in the proposition. We are going to show that
  \(\lnot\lnot P \to P\) for every proposition \(P : \V\), from which excluded
  middle in \(\V\) holds. Let \(P\) be a proposition in \(\V\) and assume that
  \(\lnot\lnot P\). This yields \(\Zero \neq P\), so by assumption the equality
  \(P = \bigvee \delta_{\Zero,P,P}\) implies \(P\). But, recalling item~(ii) of
  \cref{examples-of-delta-complete-posets}, we have exactly this equality
  \(\bigvee \delta_{0,P,P} = P\).
\end{proof}

We have seen that having a pair of elements \(x,y\) with \(x \below y\) and
\(x \neq y\) is very weak constructively.  As promised in the introduction of
this section, we now introduce a constructively stronger notion.

\begin{definition}[Strictly below, \(x \sbelow y\)]
  \label{def:strictly-below}
  Let \((X,\below)\) be a \deltacomplete{\V} poset and \(x,y : X\). We~say that
  \(x\) is \emph{strictly below} \(y\) if \(x \below y\) and, moreover, for
  every \(z \aboveorder y\) and every proposition \(P : \V\), the equality
  \(z = \bigvee \delta_{x,z,P}\) implies \(P\).
\end{definition}
Note that with excluded middle, \(x \sbelow y\) is equivalent to the conjunction
of \(x \below y\) and \(x \neq y\). But constructively, the former is much
stronger, as the following example and proposition illustrate.

\begin{example}[Strictly below in \(\Omega_{\V}\)]
  Recall from~\cref{examples-of-delta-complete-posets} that \(\Omega_{\V}\) is
  \deltacomplete{\V}. Let \(P : \V\) be an arbitrary proposition. Observe that
  \(\Zero_\V \neq P\) precisely when \(\lnot\lnot P\) holds. However,
  \(\Zero_\V\) is strictly below \(P\) if and only if \(P\) holds.
\end{example}

\begin{proposition}
  \label{sbelow-below-neq}
  For a \deltacomplete{\V} poset \((X,\below)\) and \(x,y : X\), we have that
  \(x \sbelow y\) implies both \(x \below y\) and \(x \neq y\). However, if the
  conjunction of \(x \below y\) and \(x \neq y\) implies \(x \sbelow y\) for
  every \(x,y : \Omega_\V\), then excluded middle in \(\V\) holds.
\end{proposition}
\begin{proof}
  Note that \(x \sbelow y\) implies \(x \below y\) by definition. Now suppose
  that \(x \sbelow y\) and assume \(x = y\) for a contradiction. Since we
  assumed \(x \sbelow y\), the equality \(y = \bigvee \delta_{x,y,\Zero_{\V}}\)
  implies that \(\Zero_{\V}\) holds. But this equality holds since \(x = y\) by
  our other assumption, so \(x \neq y\), as desired.

  For \(P : \Omega_{\V}\) we observed that \(\Zero_\V \neq P\) is equivalent to
  \(\lnot\lnot P\) and that \(\Zero_\V \sbelow P\) is equivalent to \(P\), so if
  we had \(\pa*{\pa*{x \below y} \times \pa*{x \neq y}} \to x \sbelow y\) in
  general, then we would have \(\lnot\lnot P \to P\) for every proposition \(P\)
  in \(\V\), which is equivalent to excluded middle in \(\V\).
\end{proof}

\begin{lemma}\label{sbelow-trans}
  Let \((X,\below)\) be a \deltacomplete{\V} poset and \(x,y,z : X\). The
  following hold:
  \begin{enumerate}[(i)]
  \item If \(x \below y \sbelow z\), then \(x \sbelow z\).
  \item If \(x \sbelow y \below z\), then \(x \sbelow z\).
  \end{enumerate}
\end{lemma}
\begin{proof}
  For (i), assume \(x \below y \sbelow z\), let \(P\) be an arbitrary
  proposition in \(\V\) and suppose that \(z \below w\). We must show that
  \(w = \bigvee \delta_{x,w,P}\) implies \(P\). But \(y \sbelow z\), so we know
  that the equality \(w = \bigvee \delta_{y,w,P}\) implies \(P\). Now observe
  that \(\bigvee \delta_{x,w,P} \below \bigvee \delta_{y,w,P}\), so if
  \(w = \bigvee \delta_{x,w,P}\), then \(w = \bigvee \delta_{y,w,P}\), finishing
  the proof.
  For (ii), assume \(x \sbelow y \below z\), let \(P\) be an arbitrary
  proposition in \(\V\) and suppose that \(z \below w\). We must show that
  \(w = \bigvee \delta_{x,w,P}\) implies \(P\). But \(x \sbelow y\) and
  \(y \below w\), so this follows immediately.
\end{proof}

\begin{proposition}
  \label{positive-element-equivalent}
  Let \((X,\below)\) be a \(\V\)-sup-lattice and let \(y : X\). The following
  are equivalent:
  \begin{enumerate}[(i)]
  \item the least element of \(X\) is strictly below \(y\);
  \item for every family \(\alpha : I \to X\) with \(I : \V\) and
    \(y \below \bigvee \alpha\), there exists some element \(i : I\).
  \item there exists some \(x : X\) with \(x \sbelow y\).
  \end{enumerate}
\end{proposition}
\begin{proof}
  Write \(\bot\) for the least element of \(X\). By~\cref{sbelow-trans} we have:
  \[
    \bot \sbelow y
    \iff \exists_{x : X}\pa*{\bot \below x \sbelow y}
    \iff \exists_{x : X}\pa*{x \sbelow y},
  \]
  which proves the equivalence of (i) and (iii). It remains to prove that (i) and
  (ii) are equivalent. Suppose that \(\bot \sbelow y\) and let
  \(\alpha : I \to X\) with \(y \below \bigvee \alpha\). Using
  \(\bot \sbelow y \below \bigvee \alpha\) and~\cref{sbelow-trans}, we have
  \(\bot \sbelow \bigvee \alpha\). Hence, we only need to prove
  \(\bigvee \alpha \below \bigvee \delta_{\bot,\bigvee \alpha,\exists {i :
      I}}\), but
  \(\alpha_j \below \bigvee \delta_{\bot,\bigvee\alpha,\exists {i : I}}\) for
  every \(j : I\), so this is true indeed.
  For the converse, assume that \(y\) satisfies (ii), suppose
  \(z \aboveorder y\) and let \(P : \V\) be a proposition such that
  \(z = \bigvee \delta_{\bot,z,P}\). We must show that \(P\) holds. But notice
  that
  \(y \below z = \bigvee \delta_{\bot,z,P} = \bigvee \pa*{(p : P)\mapsto z}\),
  so \(P\) must be inhabited as \(y\) satisfies~(ii).
\end{proof}
Item (ii) in~\cref{positive-element-equivalent} says exactly that \(y\) is a
positive element in the sense of~\cite[p.~98]{Johnstone1984}. We note that item
(iii) in~\cref{positive-element-equivalent} makes sense even when \((X,\below)\)
is not a \(\V\)-sup-lattice, but just a \deltacomplete{\V} poset. Accordingly,
we make the following definition.

\begin{definition}[Positive element]
  \label{def:positive-element}
  An element of a \deltacomplete{\V} poset is \emph{positive} if it satisfies
  item~(iii) in~\cref{positive-element-equivalent}.
\end{definition}

An element of a \(\V\)-dcpo is called \emph{compact} if it is inaccessible by
directed joins of families indexed by types in
\(\V\)~\cite[Definition~44]{deJongEscardo2021}.

\begin{proposition}
  A compact element \(x\) of a \(\V\)-dcpo with least element \(\bot\) is
  positive if and only if \(x \neq \bot\).
\end{proposition}
\begin{proof}
  One implication is taken care of by~\cref{sbelow-below-neq}. For the converse,
  suppose that \(x \neq \bot\).
  We show that \(\bot\) is strictly below \(x\). For if
  \(x \below y = \bigvee \delta_{\bot,y,P}\), then by compactness of \(x\),
  there must exist \(i : \One + P\) such that \(x \below \delta_{\bot,y,P}(i)\)
  already.  But \(i\) can't be equal to \(\inl(\star)\), since \(x\) is assumed
  to be different from~\(\bot\). Hence, \(i = \inr(p)\) and \(P\) must hold.
\end{proof}
Looking to strengthen the notion of a nontrivial poset, we make the following
definition, whose terminology is inspired by~\cref{def:positive-element}.

\begin{definition}[Positive poset]
  A \deltacomplete{\V} poset \(X\) is \emph{positive} if we have designated
  \(x,y : X\) with \(x\) strictly below \(y\).
\end{definition}

\begin{examples}\hfill
  \begin{enumerate}[(i)]
  \item Consider an element \(P\) of the \deltacomplete{\V} poset
    \(\Omega_\V\). The pair \(\pa*{\Zero_\V , P}\) witnesses nontriviality of
    \(\Omega_\V\) if and only if \(\lnot\lnot P\) holds, while it witnesses
    positivity if and only if \(P\) holds.
  \item Consider the \(\V\)-powerset \(\powerset_{\V}(X)\) on a type \(X\) as a
    \deltacomplete{\V} poset
    (recall~\cref{examples-of-delta-complete-posets}). We write
    \(\emptyset : \powerset_{\V}(X)\) for the map \(x \mapsto \Zero_\V\).  Say
    that a subset \(A : \powerset_{\V}(X)\) is nonempty if \(A \neq \emptyset\)
    and inhabited if there exists some \(x : X\) such that \(A(x)\) holds.  The
    pair \((\emptyset , A)\) witnesses nontriviality of \(\powerset_\V(X)\) if
    and only if \(A\) is nonempty, while it witnesses positivity if and only if
    \(A\) is inhabited. In particular, \(\powerset_\V(X)\) is positive if and
    only if \(X\) is an inhabited type.
  \end{enumerate}
\end{examples}

% \begin{examples}\hfill
%   \begin{enumerate}[(i)]
%   \item The \deltacomplete{\V} poset \(\Omega_{\V}\) is positive with
%     \(\Zero_{\V} \sbelow \One_{\V}\).
%   \item Let \(X : \V\) be a set with a point \(x : X\). The \deltacomplete{\V} poset
%   \(\powerset_{\V}(X)\) (recall~\cref{examples-of-delta-complete-posets}) with
%   points \(\emptyset\) and \(\{x\}\) is positive. Here we employ the familiar
%   set-theoretic notation \(\{x\}\) for what is formally the function
%   \(\pa*{y \mapsto x = y} : X \to \Omega_{\V}\), which is well-defined, because
%   \(X\) is assumed to be a set. Similarly, \(\emptyset\) formally denotes the
%   function \(y \mapsto \Zero\).

%   By contrast, if we only know that \(X\) is nonempty (i.e.\
%   \(X \neq \Zero_{\V}\)), then \(\powerset_{\V}(X)\) with points \(\emptyset\)
%   and \(X\) (considered as a subset) is only nontrivial.
%   \end{enumerate}
% \end{examples}

\subsection{Retract Lemmas}\label{sec:retract-lemmas}
We show that the type of propositions in \(\V\) is a retract of any positive
\deltacomplete{\V} poset and that the type of \(\lnot\lnot\)-stable
propositions in \(\V\) is a retract of any nontrivial \deltacomplete{\V} poset.

\begin{definition}[\(\Delta_{x,y} : \Omega_{\V} \to X\)]
  Suppose that \((X,\below,x,y)\) is a nontrivial \deltacomplete{\V} poset.
  We define \(\Delta_{x,y} : \Omega_{\V} \to X\) by the assignment
  \(P \mapsto \bigvee \delta_{x,y,P}\).
\end{definition}
We will often omit the subscripts in \(\Delta_{x,y}\) when it is clear from the
context.

\begin{definition}[Locally small]
  A \deltacomplete{\V} poset \((X,\below)\) is \emph{locally small} if its order
  has values of size \(\V\), i.e.\ we have
  \({\below_{\V}} : X \to X \to \V\) with
  \(\pa*{x \below y} \simeq \pa*{x \below_{\V} y}\) for every \(x,y : X\).
\end{definition}

\begin{examples}\hfill
  \begin{enumerate}[(i)]
  \item The \(\V\)-sup-lattices \(\Omega_{\V}\) and \(\powerset_\V(X)\) (for
    \(X : \V\)) are locally small.
  \item All examples of \(\V\)-dcpos in~\cite{deJongEscardo2021} are locally
    small.
  \end{enumerate}
\end{examples}

\begin{lemma}\label{nontrivial-retract}
  A locally small \deltacomplete{\V} poset \((X,\below)\) is nontrivial,
  witnessed by elements \(x \below y\), if~and only if the composite
  \(\Omeganotnot{\V} \hookrightarrow \Omega_{\V} \xrightarrow{\Delta_{x,y}} X\)
  is a section.
\end{lemma}
\begin{proof}
  Suppose first that \((X,\below,x,y)\) is nontrivial and locally small. We define
  \begin{align*}
    r : X &\to \Omeganotnot{\V} \\
    z &\mapsto z \not\below_{\V} x.
  \end{align*}
  Note that negated propositions are \(\lnot\lnot\)-stable, so \(r\) is
  well-defined. Let \(P : \V\) be an arbitrary \(\lnot\lnot\)-stable
  proposition. We want to show that \(r (\Delta_{x,y}(P)) = P\). By propositional
  extensionality, establishing logical equivalence suffices.
  Suppose first that \(P\) holds. Then
  \(\Delta_{x,y}(P) \equiv \bigvee \delta_{x,y,P} = y\), so
  \(r(\Delta_{x,y}(P)) = r(y) \equiv \pa*{y \not\below_{\V} x}\) holds by
  antisymmetry and our assumptions that \(x \below y\) and \(x \neq y\).
  Conversely, assume that \(r(\Delta_{x,y}(P))\) holds, i.e.\ that we have
  \(\bigvee \delta_{x,y,P} \not\below_{\V} x\). Since \(P\) is
  \(\lnot\lnot\)-stable, it suffices to derive a contradiction from~\(\lnot
  P\). So assume~\(\lnot P\). Then \(x = \bigvee \delta_{x,y,P}\), so
  \(r(\Delta_{x,y}(P)) = r(x) \equiv x \not\below_{\V} x\), which is false by
  reflexivity.

  For the converse, assume that
  \(\Omeganotnot{\V} \hookrightarrow \Omega_{\V} \xrightarrow{\Delta_{x,y}} X\)
  has a retraction \(r : \Omeganotnot{\V} \to X\). Then
  \(\Zero_{\V} = r(\Delta_{x,y}(\Zero_{\V})) = r(x)\) and
  \(\One_{\V} = r(\Delta_{x,y}(\One_{\V})) = r(y)\),
  where we used that \(\Zero_{\V}\) and \(\One_{\V}\) are \(\lnot\lnot\)-stable.
  Since \(\Zero_{\V} \neq \One_{\V}\), we get \(x \neq y\), so
  \((X,\below,x,y)\) is nontrivial, as desired.
\end{proof}
The appearance of the double negation in the above lemma is due to the
definition of nontriviality. If we instead assume a positive poset \(X\), then
we can exhibit all of \(\Omega_{\V}\) as a retract of \(X\).
\begin{lemma}\label{positive-retract}
  A locally small \deltacomplete{\V} poset \((X,\below)\) is positive, witnessed
  by elements \(x \below y\), if~and only if for every \(z \aboveorder y\), the
  map \(\Delta_{x,z} : \Omega_{\V} \to X\) is a section.
\end{lemma}
\begin{proof}
  Suppose first that \((X,\below,x,y)\) is positive and locally small and let
  \(z \aboveorder y\) be arbitrary. We define
  \begin{align*}
    r_z : X &\mapsto \Omega_{\V} \\
    w &\mapsto z \below_{\V} w.
  \end{align*}
  Let \(P : \V\) be arbitrary proposition. We want to show that
  \(r_z(\Delta_{x,z}(P)) = P\). Because of propositional extensionality, it
  suffices to establish a logical equivalence between \(P\) and
  \(r_z(\Delta_{x,z}(P))\).
  Suppose first that \(P\) holds. Then \(\Delta_{x,z}(P) = z\), so
  \(r_z(\Delta_{x,z}(P)) = r_z(z) \equiv \pa*{z \below_{\V} z}\) holds as well
  by reflexivity.
  Conversely, assume that \(r_z(\Delta_{x,z}(P))\) holds, i.e.\ that we have
  \(z \below_{\V} \bigvee \delta_{x,z,P}\). Since
  \({\bigvee \delta_{x,z,P} \below z}\) always holds, we get
  \(z = \bigvee \delta_{x,z,P}\) by antisymmetry. But by assumption
  and~\cref{sbelow-trans}, the element \(x\) is strictly~below~\(z\), so \(P\)
  must hold.

  For the converse, assume that for every \(z \aboveorder y\), the map
  \(\Delta_{x,z} : \Omega_{\V} \to X\) has a retraction
  \(r_z : X \to \Omega_{\V}\). We must show that the equality
  \(z = \Delta_{x,z}(P)\) implies \(P\) for every \(z \aboveorder y\) and
  proposition \(P : \V\). Assuming \(z = \Delta_{x,z}(P)\), we have
  \(\One_{\V} = r_z(\Delta_{x,z}(\One_{\V})) = r_z(z) = r_z(\Delta_{x,z}(P)) =
  P\), so \(P\) must hold indeed. Hence, \((X,\below,x,y)\) is positive, as
  desired.
\end{proof}

\subsection{Reductions to Impredicativity and Excluded Middle}
\label{sec:reductions}
We present our main theorems here, which show that, constructively and
predicatively, nontrivial \deltacomplete{\V} posets are necessarily large and
necessarily lack decidable equality.

\begin{definition}[Small]
  A \deltacomplete{\V} poset is \emph{small} if it is locally small and its
  carrier has size \(\V\).
\end{definition}

\begin{theorem}\label{nontrivial-impredicativity}\label{positive-impredicativity}\hfill
  \begin{enumerate}[(i)]
  \item\label{nontrivial-impredicativity-1} There is a nontrivial small
    \deltacomplete{\V} poset if and only if \(\Omeganotnotresizingalt{\V}\)
    holds.
  \item\label{positive-impredicativity-2} There is a positive small
    \deltacomplete{\V} poset if and only if \(\Omegaresizingalt{\V}\) holds.
  \end{enumerate}
\end{theorem}
\begin{proof}
  \eqref{nontrivial-impredicativity-1}
  Suppose that \((X,\below,x,y)\) is a nontrivial small \deltacomplete{\V}
    poset. By \cref{nontrivial-retract}, we can exhibit \(\Omeganotnot{\V}\) as
    a retract of \(X\). But \(X\) has size \(\V\) by assumption, so
    by~\cref{size-retract} and the fact that \(\Omeganotnot{\V}\) is a set, the
    type \(\Omeganotnot{\V}\) has size \(\V\) as~well.
    For the converse, note that
    \(\pa*{\Omeganotnot{\V},\to,\Zero_{\V},\One_{\V}}\) is a nontrivial
    \(\V\)-sup-lattice with \(\bigvee \alpha\) given by
    \(\lnot\lnot\exists_{i : I}\alpha_i\). And if we assume
    \(\Omeganotnotresizingalt{\V}\), then it is small.

    \eqref{positive-impredicativity-2}
    Suppose that \((X,\below,x,y)\) is a positive small poset. By
    \cref{positive-retract}, we can exhibit \(\Omega_{\V}\) as a retract of
    \(X\). But \(X\) has size \(\V\) by assumption, so by~\cref{size-retract}
    and the fact that \(\Omega_{\V}\) is a set, the type \(\Omega_{\V}\) has
    size \(\V\) as well.
    For the converse, note that \(\pa*{\Omega_{\V},\to,\Zero_\V,\One_\V}\) is a
    positive \(\V\)-sup-lattice. And if we assume \(\Omegaresizingalt{\V}\),
    then it is small. \qedhere
\end{proof}

\begin{lemma}[\href{https://www.cs.bham.ac.uk/~mhe/agda-new/DiscreteAndSeparated.html\#retract-is-discrete}{\texttt{retract-is-discrete}} and
  \href{https://www.cs.bham.ac.uk/~mhe/agda-new/DiscreteAndSeparated.html\#subtype-is-\%C2\%AC\%C2\%AC-separated}{\texttt{subtype-is-\(\lnot\lnot\)-separated}}
  in \cite{TypeTopology}]\hfill
  \label{equality-retract}
  \begin{enumerate}[(i)]
  \item Types with decidable equality are closed under retracts.
  \item Types with \(\lnot\lnot\)-stable equality are closed under retracts.
  \end{enumerate}
\end{lemma}

\begin{theorem}\label{nontrivial-weak-em}
  There is a nontrivial locally small \deltacomplete{\V} poset with decidable
  equality if and only if weak excluded middle in \(\V\) holds.
\end{theorem}
\begin{proof}
  Suppose that \((X,\below,x,y)\) is a nontrivial locally small
  \deltacomplete{\V} poset with decidable equality. Then
  by~\cref{nontrivial-retract,equality-retract}, the type \(\Omeganotnot{\V}\)
  must have decidable equality too. But negated propositions are
  \(\lnot\lnot\)-stable, so this yields weak excluded middle in \(\V\). For the
  converse, note that \(\pa*{\Omeganotnot{\V},\to,\Zero_\V,\One_\V}\) is a
  nontrivial \(\V\)-sup-lattice that has decidable equality if and only if weak
  excluded middle in \(\V\) holds.
\end{proof}

\begin{theorem}\label{positive-em} The following are equivalent:
  \begin{enumerate}[(i)]
  \item There is a positive locally small \deltacomplete{\V} poset with
    \(\lnot\lnot\)-stable equality.
  \item There is a positive locally small \deltacomplete{\V} poset with
    decidable equality.
  \item Excluded middle in \(\V\) holds.
  \end{enumerate}
\end{theorem}
\begin{proof}
  Note that \(\textup{(ii)}\Rightarrow\textup{(i)}\), so we are left to show
  that \(\textup{(iii)}\Rightarrow\textup{(ii)}\) and that
  \(\textup{(i)}\Rightarrow\textup{(iii)}\). For the first implication, note
  that \(\pa*{\Omega_{\V},\to,\Zero_\V,\One_\V}\) is a positive
  \(\V\)-sup-lattice that has decidable equality if and only if excluded middle
  in \(\V\) holds.
  To see that (i)~implies~(iii), suppose that \((X,\below,x,y)\) is a positive
  locally small \deltacomplete{\V} poset with \(\lnot\lnot\)-stable
  equality. Then by~\cref{positive-retract,equality-retract} the type
  \(\Omega_{\V}\) must have \(\lnot\lnot\)-stable equality. But this implies
  that \(\lnot\lnot P \to P\) for every proposition \(P\) in \(\V\) which is
  equivalent to excluded middle in \(\V\).
\end{proof}
Lattices, bounded complete posets and dcpos are necessarily large and
necessarily lack decidable equality in our predicative constructive setting.
More precisely,
\begin{corollary}\hfill
  \begin{enumerate}[(i)]
  \item There is a nontrivial small \(\V\)-sup-lattice (or \(\V\)-bounded complete
    poset or \(\V\)-dcpo) if~and~only~if \(\Omeganotnotresizingalt{\V}\) holds.
  \item There is a positive small \(\V\)-sup-lattice (or \(\V\)-bounded complete
    poset or \(\V\)-dcpo) if~and~only~if \(\Omegaresizingalt{\V}\) holds.
  \item There is a nontrivial locally small \(\V\)-sup-lattice (or
    \(\V\)-bounded complete poset or \(\V\)-dcpo) with decidable equality if and
    only if weak excluded middle in \(\V\) holds.
  \item There is a positive locally small \(\V\)-sup-lattice (or \(\V\)-bounded
    complete poset or \(\V\)-dcpo) with decidable equality if and only if
    excluded middle in \(\V\) holds.
  \end{enumerate}
\end{corollary}

\subsection{Unspecified Nontriviality and Positivity}
\label{sec:unspecified}
The above notions of non-triviality and positivity are data rather than
property. Indeed, a nontrivial poset \((X,\below)\) is (by definition) equipped
with two designated points \(x,y : X\) such that \(x \below y\) and
\(x \neq y\). It is natural to wonder if the propositionally truncated versions
of these two notions yield the same conclusions. In this section we show that
this is indeed the case if we assume univalence. The need for the univalence
assumption comes from the fact that the notion of having a given size is
property precisely if univalence holds, as shown in
\cref{has-size-is-prop,has-size-univalence}.

\begin{definition}[Nontrivial/positive in an unspecified way]
  A poset \((X,\below)\) is \emph{nontrivial in an unspecified way} if there
  exist some elements \(x,y : X\) such that \(x \below y\) and \(x \neq y\),
  i.e.\
  \(\exists_{x: X}\exists_{y : X}\pa*{\pa*{x \below y} \times \pa*{x \neq y}}\).
  Similarly, we can define when a poset is \emph{positive in an unspecified way}
  by truncating the notion of positivity.
\end{definition}

\begin{theorem}
  Suppose that the universes \(\V\) and \(\V^+\) are univalent.
  \begin{enumerate}[(i)]
  \item\label{unspecified-1} There is a small \deltacomplete{\V} poset that is
    nontrivial in an unspecified way if and only if
    \(\Omeganotnotresizingalt{\V}\) holds.
  \item\label{unspecified-2} There is a small \deltacomplete{\V} poset that is
    positive in an unspecified way if and only if \(\Omegaresizingalt{\V}\)
    holds.
  \end{enumerate}
\end{theorem}
\begin{proof}
  \eqref{unspecified-1} Suppose that \((X,\below)\) is a \deltacomplete{\V}
  poset that is nontrivial in an unspecified way.  By~\cref{has-size-is-prop}
  and univalence of \(\V\) and \(\V^+\), type \(\Omeganotnot{\V}\hassize{\V}\)
  is a proposition. By the universal property of the propositional truncation,
  in proving that \(\Omeganotnot{\V}\hassize{\V}\) we can therefore assume that
  are given points \(x,y : X\) with \(x \below y\) and \(x \neq y\). The result
  then follows from~\cref{nontrivial-impredicativity}.
  \eqref{unspecified-2} By reduction to item~(ii) of
  \cref{positive-impredicativity}.
\end{proof}
Similarly, we can prove the following theorems by reduction to
\cref{nontrivial-weak-em,positive-em}.

\begin{theorem}\hfill
  \begin{enumerate}[(i)]
  \item There is a locally small \deltacomplete{\V} poset with decidable
    equality that is nontrivial in an unspecified way if and only if weak
    excluded middle in \(\V\) holds.
  \item There is a locally small \deltacomplete{\V} poset with decidable
    equality that is positive in an unspecified way if and only if excluded
    middle in \(\V\) holds.
  \end{enumerate}
\end{theorem}

\section{Maximal Points and Fixed Points}
\label{sec:maximal-and-fixed-points}
In this section we construct a particular example of a \(\V\)-sup-lattice that
will prove very useful in studying the predicative validity of some well-known
principles in order theory.

\begin{definition}[Lifting, cf.~\cite{EscardoKnapp2017}]
  \label{lifting-of-prop}
  Fix a proposition \(P_\U\) in a universe \(\U\). Lifting \(P_\U\) with respect
  to a universe \(\V\) is defined by
  \[
    \lifting_{\V}\pa*{P_\U} \colonequiv \sum_{Q : \Omega_\V} \pa*{Q \to P_\U}.
  \]
\end{definition}
This is a subtype of \(\Omega_\V\) and it is closed under \(\V\)-suprema (in
particular, it contains the least~element).

\begin{examples}\hfill
  \begin{enumerate}[(i)]
  \item If \(P_\U \colonequiv \Zero_\U\), then
    \( \lifting_{\V}(P_\U) \simeq \pa*{\sum_{Q : \Omega_\V} \lnot Q} \simeq
    \pa*{\sum_{Q : \Omega_\V} Q = \Zero_{\V}} \simeq \One \).
  \item If \(P_\U \colonequiv \One_\U\), then
    \( \lifting_{\V}(P_\U) \equiv \pa*{\sum_{Q : \Omega_\V} \pa*{Q \to \One_\U}}
    \simeq \Omega_\V \).
  \end{enumerate}
\end{examples}
What makes \(\lifting_{\V}(P_\U)\) useful is the following observation.
\begin{lemma}\label{maximal-iff-resize}
  Suppose that the poset \(\lifting_{\V}(P_\U)\) has a maximal element
  \(Q : \Omega_\V\). Then \(P_\U\) is equivalent to \(Q\), which is the
  greatest element of \(\lifting_{\V}(P_\U)\). In particular,
  \(P_\U\)~has~size~\(\V\).
  Conversely, if \(P_\U\) is equivalent to a proposition \(Q : \Omega_\V\), then
  \(Q\) is the greatest element of~\(\lifting_{\V}(P_\U)\).
\end{lemma}
\begin{proof}
  Suppose that \(\lifting_{\V}(P_\U)\) has a maximal element \(Q :
  \Omega_\V\). We wish to show that \(Q \simeq P_\U\). By definition of
  \(\lifting_{\V}(P_\U)\), we already have that \(Q \to P_\U\). So only the
  converse remains. Therefore suppose that \(P_\U\) holds. Then, \(\One_\V\) is
  an element of \(\lifting_{\V}(P_\U)\). Obviously \(Q \to 1_\V\), but \(Q\) is
  maximal, so actually \(Q = 1_\V\), that is, \(Q\) holds, as
  desired. Thus, \(Q \simeq P_\U\). It~is then straightforward to see that \(Q\)
  is actually the greatest element of \(\lifting_{\V}(P_\U)\), since
  \(\lifting_{\V}(P_\U) \simeq \sum_{Q' : \Omega_\V}(Q' \to Q)\).
  For the converse, assume that \(P_\U\) is equivalent to a proposition
  \(Q : \Omega_\V\). Then, as before,
  \(\lifting_{\V}(P_\U) \simeq \sum_{Q' : \Omega_\V}(Q' \to Q)\), which shows
  that \(Q\) is indeed the greatest element of~\(\lifting_{\V}(P_\U)\).
\end{proof}

\begin{corollary}
  Let \(P_{\U}\) be a proposition in \(\U\). The \(\V\)-sup-lattice
  \(\lifting_{\V}(P_{\U})\) has all \(\V\)-infima if and only if
  \(P_{\U}\)~has~size~\(\V\).
\end{corollary}
\begin{proof}
  Suppose first that \(\lifting_{\V}(P_{\U})\) has all \(\V\)-infima. Then it
  must have a infimum for the empty family
  \(\Zero_{\V} \to \lifting_{\V}(P_{\U})\). But this infimum must be the greatest
  element of \(\lifting_{\V}(P_{\U})\). So by \cref{maximal-iff-resize} the
  proposition \(P_{\U}\) must have size \(\V\).

  Conversely, suppose that \(P_{\U}\) is equivalent to a proposition \(Q : \V\).
  Then the infimum of a family \(\alpha : I \to \lifting_{\V}(P_{\U})\) with
  \(I : \V\) is given by \(\pa*{Q \times \Pi_{i : I} \alpha_i} : \V\).
\end{proof}

\begin{definition}[\(\Zorn{\V}{\U}{\T}\)]
  Let \(\U\), \(\V\) and \(\T\) be
  universes. \(\text{\emph{Zorn's-Lemma}}_{\V,\U,\T}\) asserts that every
  pointed \(\V\)-dcpo with carrier in \(\U\) and order taking values in \(\T\)
  (cf.\ \cite{deJongEscardo2021}) has a maximal element.
\end{definition}
It important to note that Zorn's lemma does \emph{not} imply the Axiom of Choice
in the absence of excluded middle~\cite{Bell1997}. If it did, then the following
would be useless, since the Axiom of Choice implies excluded middle, which in
turn implies propositional resizing.
\begin{theorem}\label{Zorn-implies-Propositional-Resizing}
  \(\Zorn{\V}{\V^+\sqcup\U}{\V}\) implies \(\Propresizing{\U}{\V}\).
\end{theorem}
In particular, \(\Zorn{\V}{\V^+}{\V}\) implies \(\Propresizing{\V^+}{\V}\).
\begin{proof}
  Suppose that \(\Zorn{\V}{\V^+\sqcup\U}{\V}\) were true. Then
  \(\lifting_{\V}(P) : \V^+\sqcup\U\) has a maximal element for every
  \(P : \Omega_{\U}\). Hence, by \cref{maximal-iff-resize}, every
  \(P : \Omega_{\U}\) has size \(\V\).
\end{proof}

We can also use \cref{maximal-iff-resize} to show that the following version of
Tarski's fixed point theorem \cite{Tarski1955} is not available predicatively.
\begin{definition}[\(\Tarski{\V}{\U}{\T}\)]
  The assertion \(\text{\emph{Tarski's-Theorem}}_{\V,\U,\T}\) says that every
  monotone endofunction on a \(\V\)-sup-lattice with carrier in a universe
  \(\U\) and order taking values in a universe \(\T\) has a greatest fixed
  point.
\end{definition}
\begin{theorem}
  \(\Tarski{\V}{\V^+\sqcup\U}{\V}\) implies \(\Propresizing{\U}{\V}\).
\end{theorem}
In particular, \(\Tarski{\V}{\V^+}{\V}\) implies \(\Propresizing{\V^+}{\V}\).
\begin{proof}
  Suppose that \(\Tarski{\V}{\V^+\sqcup\U}{\V}\) were true and let
  \(P : \Omega_{\U}\) be arbitrary. Consider the \(\V\)-sup-lattice
  \(\lifting_{\V}(P) : \V^+ \sqcup \U\). By assumption, the identity map on this
  poset has a greatest fixed point, but this must be the greatest element of
  \(\lifting_{\V}(P)\), which implies that \(P\)~has~size~\(\V\) by
  \cref{maximal-iff-resize}.
\end{proof}

Another famous fixed point theorem, for dcpos this time, is due to
Pataraia~\cite{Pataraia1997,Escardo2003} which says that every monotone
endofunction on a pointed dcpo has a least fixed point. (A dcpo is called
pointed if it has a least element.)
A crucial step in proving Pataraia's theorem is the observation that every dcpo
has a greatest monotone inflationary endofunction. (An endomap \(f : {X \to X}\)
is inflationary when \(x \below f(x)\) for every \(x : X\).)  We refer to this
intermediate result as Pataraia's lemma.

\begin{definition}[\(\Pataraia{\V}{\U}{\T}\), \(\PataraiaThm{\V}{\U}{\T}\)]
  \hfill
  \begin{enumerate}[(i)]
  \item \(\PataraiaThm{\V}{\U}{\T}\) says that every monotone endofunction on a
    pointed \(\V\)-dcpo with carrier in a universe \(\U\) and order taking
    values in a universe \(\T\) has a least fixed~point.
  \item \(\text{\emph{Pataraia's-Lemma}}_{\V,\U,\T}\) says that every
    \(\V\)-dcpo with carrier in a universe \(\U\) and order taking values in a
    universe \(\T\) has a greatest monotone inflationary endofunction.
  \end{enumerate}
\end{definition}
A careful analysis of the proof in~\cite[Section~2]{Escardo2003} shows that in
our predicative setting we can still prove that
\(\Pataraia{\V}{\U\sqcup\T}{\U\sqcup\T}\) implies
\(\PataraiaThm{\V}{\U}{\T}\). However, Pataraia's lemma is not available
predicatively.
\begin{theorem}\label{Pataraia-implies-Propositional-Resizing}
  \(\Pataraia{\V}{\V^+\sqcup\U}{\V}\) implies \(\Propresizing{\U}{\V}\).
\end{theorem}
In particular, \(\Pataraia{\V}{\V^+}{\V}\) implies \(\Propresizing{\V^+}{\V}\).
\begin{proof}
  Suppose that \(\Pataraia{\V}{\V^+\sqcup\U}{\V}\) were true and let
  \(P : \Omega_{\U}\) be arbitrary. Consider the \(\V\)-dcpo
  \(\lifting_{\V}(P) : \V^+ \sqcup \U\). By assumption, it has a greatest
  monotone inflationary endomap \(g : \lifting_{\V}(P) \to \lifting_{\V}(P)\).
  We claim that \(g(\Zero_{\V})\) is a maximal element of \(\lifting_{\V}(P)\),
  which would finish the proof by \cref{maximal-iff-resize}. So suppose that we
  have \(Q : \lifting_{\V}(P)\) with \(g\pa*{\Zero_\V} \below Q\). Then we must
  show that \(Q \below g\pa*{\Zero_\V}\). Define
  \(f_Q : \lifting_{\V}(P) \to \lifting_{\V}(P)\) by \(Q' \mapsto Q' \vee
  Q\). Note that \(f_Q\) is monotone and inflationary, so that \(f_Q \below g\).
  Hence, \(Q = f_Q\pa*{\Zero_\V} \below g\pa*{\Zero_\V}\), as desired.
\end{proof}

\begin{remark}
  For a \emph{single} universe \(\U\), the usual proofs (see
  resp.~\cite{Tarski1955} and ~\cite[Section~2]{Escardo2003}) of
  \(\Tarski{\U}{\U}{\U}\), \(\Pataraia{\U}{\U}{\U}\) and (hence)
  \(\PataraiaThm{\U}{\U}{\U}\) are also valid in our predicative setting.
  However, in light of \cref{nontrivial-impredicativity}, these statements are
  not useful predicatively, because one would never be able to find interesting
  examples of posets to apply the statements to.
\end{remark}
Finally, we note that Zorn's lemma implies Pataraia's lemma with the following
universe parameters. Together with
\cref{Pataraia-implies-Propositional-Resizing} this yields another proof
that \(\Zorn{\V}{\V^+}{\V}\) implies \(\Propresizing{\V^+}{\V}\).
\begin{lemma}
  \(\Zorn{\V}{\U\sqcup\T}{\U\sqcup\T}\) implies \(\Pataraia{\V}{\U}{\T}\).
\end{lemma}
\begin{proof}
  Assume \(\Zorn{\V}{\U\sqcup\T}{\U\sqcup\T}\) and let \(D : \U\) be \(\V\)-dcpo
  with order taking values in~\(\T\). Consider the type \(\MI_D\) of monotone
  and inflationary endomaps on \(D\). We can order these maps pointwise to get a
  \(\V\)-dcpo with carrier and order taking values in \(\U\sqcup\T\). Finally,
  \(\MI_D\) has a least element: the identity map.  Hence, by our assumption, it
  has a maximal element \(g : D \to D\). It remains to show that \(g\) is in
  fact the greatest element. To this end, let \(f : D \to D\) be an arbitrary
  monotone inflationary endomap on \(D\). We must show that \(f \below
  g\). Since \(f\) is inflationary, we have \(g \below f \circ g\). So by
  maximality of \(g\), we get \(g = f \circ g\). But \(f\) is monotone and \(g\)
  is inflationary, so \(f \below f \circ g = g\), finishing the proof.
\end{proof}
The answer to the question whether Pataraia's theorem (or similarly, a least
fixed point theorem version of Tarki's theorem) is inherently impredicative or
(by contrast) does admit a predicative proof has eluded us thus far.

\section{Families and Subsets}
\label{sec:families-and-subsets}
In traditional impredicative foundations, completeness of posets is usually
formulated using subsets. For instance, dcpos are defined as posets \(D\) such
that every directed subset \(D\) has a supremum in
\(D\). \cref{examples-of-delta-complete-posets} are all formulated using small
families instead of subsets. While subsets are primitive in set theory, families
are primitive in type theory, so this could be an argument for using families
above. However, that still leaves the natural question of how the family-based
definitions compare to the usual subset-based definitions, especially in our
predicative setting, unanswered. This section aims to answer this question. We
first study the relation between subsets and families predicatively and then
clarify our definitions in the presence of impredicativity.
In our answers we will consider sup-lattices, but similar arguments could be
made for posets with other sorts of completeness, such as dcpos.

\paragraph*{All Subsets} We first show that simply asking for completeness
w.r.t.\ all subsets is not satisfactory from a predicative viewpoint.  In fact,
we will now see that even asking for all subsets \(X \to \Omega_{\T}\) for some
fixed universe \(\T\) is problematic from a predicative standpoint.

\begin{theorem}
  \label{all-T-subsets-resizing}
  Let \(\U\) and \(\V\) be universes and fix a proposition
  \(P_{\U} : \U\). Recall \(\lifting_{\V}(P_{\U})\) from~\cref{lifting-of-prop},
  which has \(\V\)-suprema.
  Let \(\T\) be any type universe.  If \(\lifting_{\V}(P_{\U})\) has suprema for
  all subsets \({\lifting_{\V}(P_U) \to \Omega_{\T}}\), then \(P_{\U}\) has size
  \(\V\) independently of \(\T\).
\end{theorem}
\begin{proof}
  Let \(\T\) be a type universe and consider the subset \(S\) of
  \(\lifting_{\V}(P_\U)\) given by \(Q \mapsto \One_\T\).
  Note that \(S\) has a supremum in \(\lifting_{\V}(P_\U)\) if and only if
  \(\lifting_{\V}(P_\U)\) has a greatest element, but
  by~\cref{maximal-iff-resize}, the latter is equivalent to \(P_\U\) having size
  \(\V\).
\end{proof}

\paragraph*{All Subsets Whose Total Spaces Have Size \(\V\)}
The proof above illustrates that if we have a subset
\(S : {X \to \Omega_{\T}}\), then there is no reason why the total space
\(\sum_{x : X} x \in S \colonequiv \sum_{x : X}\pa*{S(x) \text{ holds}}\) should
have size \(\T\). In fact, for \(S(x) \colonequiv \One_{\T}\) as above, the
latter is equivalent to asking that \(X\) has size \(\T\).

\begin{definition}[Total space of a subset, \(\totalspace\)]
  Let \(\T\) be a universe, \(X\) a type and \(S : {X \to \Omega_{\T}}\) a
  subset of \(X\). The \emph{total space} of \(S\) is defined as
  \(
    \totalspace(S) \colonequiv \sum_{x : X} x \in S.
  \)
\end{definition}
A naive attempt to solve the problem described in \cref{all-T-subsets-resizing}
would be to stipulate that a \(\V\)-sup-lattice \(X\) should have suprema for
all subsets \(S : {X \to \Omega_{\V}}\) for which \(\mathbb T(S)\)
has~size~\(\V\). Somewhat less naively, we might be more liberal and ask for
suprema of subsets \(S : {X \to \Omega_{\U\sqcup\V}}\) for which
\(\mathbb T(S)\) has size \(\V\). Here the carrier of \(X\) is in a universe
\(\U\).
Perhaps surprisingly, even this more liberal definition is too weak to be useful
as the following example shows.

\begin{example}[Naturally occurring subsets whose total spaces are not necessarily small]
  \label{natural-example-total-space}
  Let~\(X\) be a poset with carrier in \(\U\)
  and suppose that it has suprema for all (directed) subsets
  \(S : {X \to \Omega_{\U\sqcup\V}}\) for which \(\totalspace(S)\) has
  size \(\V\).
  Now let \(f : X \to X\) be a Scott continuous endofunction on \(X\). We would
  want to construct the least fixed point of \(f\) as the supremum of the
  directed subset \(S \colonequiv \{\bot,f(\bot),f^2(\bot),\dots\}\). Now, how
  do we show that its total space
  \(\totalspace(S) \equiv \sum_{x : X} \pa*{\exists_{n : \Nat}\,x = f^n(\bot)}\)
  has size \(\V\)? A first guess might be that \(\Nat \simeq \totalspace(S)\),
  which would do the job. However, it's possible that
  \(f^m(\bot) = f^{m+1}(\bot)\) for some natural number~\(m\), which would mean
  that \(\totalspace(S) \simeq \Fin(m)\) for the least such \(m\).  The problem
  is that in the absence of decidable equality on \(X\) we might not be able to
  decide which is the case. But \(X\) seldom has decidable equality, as we saw
  in~\cref{nontrivial-weak-em,positive-em}.
\end{example}

\begin{remark}
  The example above also makes clear that it is undesirable to impose an
  injectivity condition on families, as the family
  \(\Nat \to X, n \mapsto f^n(\bot)\) is not necessarily injective.
  In fact, for every type \(X : \U\) there is an equivalence between embeddings
  \(I \hookrightarrow X\) with \(I : \V\) and subsets of \(X\) whose total
  spaces have size \(\V\), cf.\
  \cite[\href{https://www.cs.bham.ac.uk/~mhe/agda-new/Slice.html\#\%F0\%9D\%93\%95-equiv}{\texttt{Slice.html}}]{TypeTopology}.
\end{remark}

\paragraph*{All \(\V\)-covered Subsets}
The point of \cref{natural-example-total-space} is analogous to the difference
between Bishop finiteness and Kuratowski finiteness. Inspired by this, we make
the following definition.

\begin{definition}[\(\V\)-covered subset]
  Let \(X\) be a type, \(\T\) a universe and \(S : X \to \Omega_\T\) a subset of
  \(X\). We say that \(S\) is \emph{\(\V\)-covered} for a universe \(\V\) if we
  have a type \(I : \V\) with a surjection \(e : I \surj \totalspace(S)\).
\end{definition}

In the example above, the subset
\(S \colonequiv \{\bot,f(\bot),f^2(\bot),\dots\}\) is \(\U_0\)-covered, because
\(\Nat \surj \totalspace(S)\).

\begin{theorem}
  \label{family-subset-equiv}
  For \(X : \U\) and any universe \(\V\) we have an equivalence between
  \(\V\)-covered subsets \(X \to \Omega_{\U \sqcup \V}\) and families
  \(I \to X\) with \(I : \V\).
\end{theorem}
\begin{proof}
  The forward map \(\varphi\) is given by \((S,I,e) \mapsto (I,\fst \circ e)\).
  In the other direction, we define \(\psi\) by mapping \((I,\alpha)\) to the
  triple \((S,I,e)\) where \(S\) is the subset of \(X\) given by
  \(S(x) \colonequiv \exists_{i : I}\,x = \alpha(i)\) and
  \(e : I \surj \totalspace(S)\) is defined as
  \(e (i) \colonequiv \pa*{\alpha(i),\tosquash*{(i,\refl)}}\).
  The composite \(\varphi \circ \psi\) is easily seen to be equal to the
  identity. To show that \(\psi \circ \varphi\) equals the identity, we need the
  following intermediate result, which is proved using function extensionality
  and path induction.
  \begin{claim*}
    Let \(S,S' : X \to \Omega_{\U\sqcup\V}\), \(e : I \to \totalspace(S)\)
    and \(e' : I \to \totalspace(S')\). If \(S = S'\) and \({{\fst} \circ e \sim
    {\fst} \circ e'}\), then \((S,e) = (S',e')\).
  \end{claim*}
  The result then follows from the claim using function extensionality and
  propositional extensionality.
\end{proof}

\begin{corollary}
  \label{family-subset-sup-equiv}
  Let \(X\) be a poset with carrier in \(\U\) and let \(\V\) be any universe.
  Then \(X\) has suprema for all \(\V\)-covered subsets
  \(X \to \Omega_{\U\sqcup\V}\) if and only if \(X\) has suprema for all
  families \(I \to X\) with \(I : \V\).
\end{corollary}

\paragraph*{Families and Subsets in the Presence of Impredicativity}
Finally, we compare our family-based approach to the
subset-based approach in the presence of impredicativity.

\begin{theorem}
  \label{impred-comparison}
  Assume \(\Omegaresizing{\T}{\U_0}\) for every universe \(\T\).
  Then the following are equivalent for a poset \(X\) in a
  universe \(\U\):
  \begin{enumerate}[(i)]
  \item \(X\) has suprema for all subsets;
  \item \(X\) has suprema for all \(\U\)-covered subsets;
  \item \(X\) has suprema for all subsets whose total spaces have size \(\U\);
  \item \(X\) has suprema for all families \(I \to X\) with \(I : \U\).
  \end{enumerate}
\end{theorem}
\begin{proof}
  Clearly \(\textup{(i)} \Rightarrow \textup{(ii)} \Rightarrow
  \textup{(iii)}\). We show that (iii) implies (i), which proves the equivalence
  of (i)--(iii). Assume that \(X\) has suprema for all subsets whose total
  spaces have size \(\U\) and let \(S : X \to \Omega_{\T}\) be any subset of
  \(X\). Using \(\Omegaresizing{\T}{\U_0}\), the total space \(\totalspace(S)\)
  has size \(\U\). So~\(X\)~has a supremum for \(S\) by assumption, as
  desired. Finally, (ii) and (iv) are equivalent
  by~\cref{family-subset-sup-equiv}.
\end{proof}
Notice that (iv) in~\cref{impred-comparison} implies that \(X\) has suprema for
all families \(I \to X\) with \(I : \V\) and \(\V\) such that
\(\V\sqcup\U \equiv \U\). Typically, in the examples
of~\cite{deJongEscardo2021}~for~instance, \(\U \colonequiv \U_1\) and
\(\V \colonequiv \U_0\), so that \(\V\sqcup\U \equiv \U\) holds. Thus, our
\(\V\)-families-based approach generalizes the traditional subset-based
approach.

\section{Conclusion}
\label{sec:conclusion}
Firstly, we have shown, constructively and predicatively, that nontrivial dcpos,
bounded complete posets and sup-lattices are all necessarily large and
necessarily lack decidable equality.
We did so by deriving a weak
impredicativity axiom or weak excluded middle from the assumption that such
nontrivial structures are small or have decidable equality,
respectively. Strengthening nontriviality to the (classically equivalent)
positivity condition, we derived a strong impredicativity axiom and
full excluded middle.

Secondly, we proved that Zorn's lemma, Tarski's greatest fixed point theorem and
Pataraia's lemma all imply impredicativity axioms. Hence, these principles are
inherently impredicative and a predicative development of order theory (in
univalent foundations) must thus do without them.

Thirdly, we clarified, in our predicative setting, the relation between the
traditional definition of a lattice that requires completeness with respect to
subsets and our definition that asks for completeness with respect to small
families.

In future work, we wish to study the predicative validity of
Pataraia's theorem and Tarski's \emph{least} fixed point theorem.
Curi~\cite{Curi2015,Curi2018} develops predicative versions of Tarki's
fixed point theorem in some extensions of CZF. It is not clear whether
these arguments could be adapted to univalent foundations, because
they rely on the set-theoretical principles discussed in the
introduction. The availability of such fixed-point theorems would be
especially useful for application to inductive sets~\cite{Aczel1977},
which we might otherwise introduce in univalent foundations using
higher inductive types~\cite{HoTTBook}.
In another direction, we have developed a
notion of apartness~\cite{BridgesVita2011} for continuous
dcpos~\cite{deJongEscardo2021} that is related to the notion of being
strictly below introduced in this paper. Namely, if \(x \below y\) are
elements of a continuous dcpo, then \(x\) is strictly below \(y\) if
\(x\) is apart from \(y\). In upcoming work, we give a constructive
analysis of the Scott topology~\cite{GierzEtAl2003} using this notion
of apartness.

\bibliography{references}

\end{document}